\documentclass[12pt]{article}

\usepackage{latexsym}
\usepackage{a4wide}
\usepackage{amscd}
\usepackage{graphics}
\usepackage{amsmath}
\usepackage{amssymb}
\usepackage{mathrsfs}
\usepackage{amsthm}
\usepackage{bbm}
\usepackage{color}
\usepackage{accents}
\usepackage{enumerate}
\input xy
\xyoption{all}

\newcommand{\N}{\mathbb{N}}                     
\newcommand{\Z}{\mathbb{Z}}                     
\newcommand{\R}{\mathbb{R}}                     
\newcommand{\T}{\mathbb{T}}                     

\newcommand{\set}[2]{\left\{{#1}\mid{#2}\right\}}       

\newtheorem{mainthm}{\sc Theorem}           
\newtheorem{thm}{\sc Theorem}[section]               
\newtheorem*{thm*}{\sc Theorem}               
\newtheorem{cor}[thm]{\sc Corollary}        
\newtheorem*{cor*}{\sc Corollary}        
\newtheorem{lem}[thm]{\sc Lemma}            
\newtheorem{prop}[thm]{\sc Proposition}     
\newtheorem{defn}[thm]{\sc Definition}      
\newtheorem{rem}[thm]{\sc Remark}           
\newtheorem{ex}[thm]{\sc Example}           

\title{Chain recurrence, chain transitivity, Lyapunov functions and rigidity of Lagrangian submanifolds of optical hypersurfaces}

\author{\hspace{-1cm}{\sc Alberto Abbondandolo${\, }^{1}$\quad Olga Bernardi${\,
}^{2}$ \quad Franco Cardin${\,
}^{2}$                  }
\vspace{0.3cm}
\\
\hspace{-1cm}${\ }^{1}$ Ruhr Universit\"at Bochum - Fakult\"at f\"ur Mathematik,
\\
\hspace{-1cm}Geb\"aude NA 4/33  \ D-44801 Bochum, Germany
\\ 
\\
\hspace{-1cm}${\ }^{2}$ Dipartimento di Matematica,
Universit\`a di Padova,\\
\hspace{-1cm}Via Trieste, 63 - 35121 Padova, Italy
\\ 
}

\date{ }

\begin{document}

\maketitle

\begin{abstract}
\noindent The aim of this paper is twofold. On the one hand, we discuss the notions of strong chain recurrence and strong chain transitivity for flows on metric spaces, together with their characterizations in terms of rigidity properties of Lipschitz Lyapunov functions. This part extends to flows some recent results for homeomorphisms of Fathi and Pageault. 
On the other hand, we use these characterisations to revisit the proof of a theorem of Paternain, Polterovich and Siburg concerning the inner rigidity of a Lagrangian submanifold $\Lambda$ contained in an optical hypersurface of a cotangent bundle, under the assumption that the dynamics on $\Lambda$ is strongly chain recurrent. We also prove an outer rigidity result for such a  Lagrangian submanifold $\Lambda$, under the stronger assumption that the dynamics on $\Lambda$ is strongly chain transitive.
\end{abstract}

\section*{Introduction}
Let $\psi=\{\psi_t\}_{t\in \R}$ be a (continuous) flow on the metric space $(X,d)$. 
We recall that a {\em strong $(\epsilon,T)$-chain} from $x$ to $y$ is a finite sequence $(x_i,t_i)_{i=1,\dots,n}$ such that $t_i\geq T$ for every $i$, $x_1=x$ and, setting $x_{n+1}:=y$, we have
\[
\sum_{i=1}^n d(\psi_{t_i}(x_i),x_{i+1}) < \epsilon.
\]
The flow $\psi$ is said to be {\em strongly chain recurrent} if for every $x\in X$, every $\epsilon>0$ and every $T\geq 0$ there exists a strong $(\epsilon,T)$-chain from $x$ to $x$. The flow $\psi$ is said to be {\em strongly chain transitive} if for every $x,y\in X$, every $\epsilon>0$ and every $T\geq 0$ there exists a strong $(\epsilon,T)$-chain from $x$ to $y$. Strong chain transitivity is a strictly stronger condition than strong chain recurrence. These notions sharpen the usual notions of chain recurrence and chain transitivity, in which one only requires each distance $d(\psi_{t_i}(x_i),x_{i+1})$ to be smaller than $\epsilon$. To the best of our knowledge, strong $(\epsilon,T)$-chains appear for the first time in the work of Easton \cite{eas78}. 

A function $h:X \rightarrow \R$ is said to be a
{\em Lyapunov function} for the flow $\psi$ on $X$ if $h\circ \psi_t \leq h$ for every $t\geq 0$, and is said to be a {\em first integral} if $h\circ \psi_t = h$ for every $t\in \R$.
Under a mild Lipschitz regularity assumption on the flow, strong chain recurrence and strong chain transitivity can be characterized in terms of rigidity properties of Lipschitz continuous Lyapunov functions. Indeed, we shall prove the following result:

\begin{mainthm}
\label{main1}
Let $\psi$ be a flow on the metric space $(X,d)$ such that $\psi_t$ is Lipschitz continuous for every $t\geq 0$, uniformly for $t$ in compact subsets of $[0,+\infty)$. Then:
\begin{enumerate} [(i)] 
\item $\psi$ is strongly chain recurrent if and only if every Lipschitz continuous Lyapunov function is a first integral.
\item $\psi$ is strongly chain transitive if and only if every Lipschitz continuous Lyapunov function is constant.
\end{enumerate}
\end{mainthm}

The Lipschitz regularity assumption is satisfied, for instance, by the flow of a Lipschitz continuous vector field on a compact manifold. 
See Propositions \ref{SCR}, \ref{pierrecont}, Theorem \ref{pierrethm}, Propositions \ref{ppp1} and \ref{ppp2} below for more precise results. See also Examples \ref{cantor} and \ref{cantor1} for a discussion on the optimality of the assumptions and 
Remark \ref{remconley} for a comparison with Conley's well known result about the characterization of the chain recurrent set in terms of continuous Lyapunov functions. Statement (i) is the flow analogue of Fathi's and Pageault's results for homeomorphisms (see \cite{fp15} and \cite{pag11}[Section 2.4]). The proof of the non trivial implications in both (i) and (ii) - that is, the construction of non-trivial Lipschitz continuous Lyapunov functions for flows which are either not strongly chain recurrent or not strongly chain transitive - uses the techniques introduced by Fathi and Pageault. Statement (ii) easily implies that flows on compact metric spaces which are ergodic with respect to a measure which is positive on every non-empty open set are strongly chain transitive, see Proposition \ref{ergodic} below. See also \cite{eas78,zhe98} for the relationship between strong chain transitivity and Lipschitz ergodicity (in the case of homeomorphisms).

In order to describe the results about rigidity of Lagrangian submanifolds, we need to recall some definitions. We consider the cotangent bundle $T^*M$ of a closed smooth manifold $M$ endowed with its standard Liouville one-form $\lambda$ and symplectic two-form $\omega=d\lambda$, which in local cotangent coordinates $(x^1,\dots,x^n, y_1,\dots,y_n)$ have the expressions
\[
\lambda = \sum_{j=1}^n y_j\, dx_j, \qquad \omega = \sum_{j=1}^n dy_j \wedge dx^j.
\]
A Lagrangian submanifold $\Lambda$ of $T^*M$ which is smoothly isotopic to the zero section through a path of Lagrangian submanifolds carries a {\em Liouville class} $\mathrm{Liouville}(\Lambda)$, which is an element of the first De Rham cohomology group $H^1(M)$. This class is defined by restricting the one-form $\lambda$ to $\Lambda$ and by pulling it back to $M$ by using the Lagrangian isotopy (see Section \ref{liouvillesec} for more details).   

A smooth function $H: T^*M \rightarrow \R$ is said to be a {\em Tonelli Hamiltonian} if it is fiberwise superlinear and its fiberwise second differential is everywhere positive definite.
An optical hypersurface in $T^*M$ is a smooth hypersurface $\Sigma\subset T^*M$ which projects surjectively onto $M$ and can be seen as a regular level set of a smooth Tonelli Hamiltonian $H: T^*M \rightarrow \R$. The optical surface $\Sigma = H^{-1}(c)$ is invariant with respect to the Hamiltonian flow $\psi^H$ of $H$. If one changes the defining Hamiltonian $H$, the flow on $\Sigma$ changes by a time reparametrization. The dynamical concepts which we consider here - mainly strong chain recurrence and strong chain transitivity - are invariant with respect to time reparametrizations, and hence depend only on the geometry of $\Sigma$. However, for sake of concreteness we shall fix once and for all a defining Hamiltonian $H$ and deal with its Hamiltonian flow $\psi^H$. 

A Lagrangian submanifold $\Lambda$ which is contained in the hypersurface $\Sigma:= H^{-1}(c)$ is automatically invariant with respect to $\psi^H$. The example of Ma\~n\'e's Hamiltonians shows that the dynamics on $\Lambda$ can be fully arbitrary: if $Y$ is any smooth vector field on $M$, the Tonelli Hamiltonian
\[
H: T^*M \rightarrow \R, \qquad H(x,y) := \frac{1}{2} \|y\|^2 + \langle y, Y(x) \rangle,
\]
where $\|\cdot\|$ denotes the norm on $T^*M$ induced by a Riemannian metric on $M$ and $\langle \cdot,\cdot \rangle$ denotes the duality pairing, has zero as a regular value, and the restriction of the flow $\psi^H$ to the zero section $\Lambda \subset \Sigma:= H^{-1}(0)$ is the flow of $Y$.

Given an optical hypersurface $\Sigma= H^{-1}(c)$, we denote by
\[
U_{\Sigma} := \set{ z\in T^*M}{H(z)<c}
\]
the precompact open subset which is bounded by $\Sigma$. In \cite{pps03}, Paternain, Polterovich and Siburg proved the following rigidity result for Lagrangian submanifolds of optical hypersurfaces having a strongly chain recurrent dynamics:

\begin{mainthm}[\cite{pps03}, Theorem 5.2]
\label{main2}
Let $\Sigma=H^{-1}(c) \subset T^*M$ be an optical hypersurface as above. Let $\Lambda$ be  a Lagrangian submanifold of $T^* M$ smoothly Lagrangian-isotopic to the zero section and contained in $\Sigma$. Assume that the restriction of the Hamiltonian flow $\psi^H$ to $\Lambda$
is strongly chain recurrent. Then if $K$ is a  Lagrangian submanifold smoothly isotopic to the zero section contained in $\overline{U_{\Sigma}}$ and having the same Liouville class of $\Lambda$, then necessarily $K=\Lambda$.
\end{mainthm}

In Section \ref{rigiditysec} we shall revisit the proof of this result, stressing the role of Lipschitz continuous Lyapunov functions and of the characterization of strong chain transitivity given by Theorem \ref{main1} (i). We shall also show that this rigidity result might fail if one assumes only chain recurrence and discuss the connections with Aubry-Mather theory.

The above theorem says that, when the dynamics on $\Lambda$ is strongly chain recurrent, one cannot find another Lagrangian submanifold in $\overline{U_{\Sigma}}$ having the same Liouville class. However, it is in general possible to find a Lagrangian submanifold with the same Liouville class outside of $U_{\Sigma}$, and actually it is also possible to obtain such a submanifold by an analytic deformation in the complement of $U_{\Sigma}$ (see Examples \ref{no rigidity} and \ref{terzo esempio} below). The following results says that the latter fact is not possible if we assume more on the dynamics on $\Lambda$, namely strong chain transitivity:

\begin{mainthm} 
\label{main3}
Let $\Sigma=H^{-1}(c) \subset T^*M$ be an optical hypersurface as above. Let $\Lambda$ be  a Lagrangian submanifold of $T^* M$ smoothly Lagrangian-isotopic to the zero section and contained in $\Sigma$. Assume that the restriction of the Hamiltonian flow $\psi^H$ to $\Lambda$
is strongly chain transitive. Let $\{\Lambda_r\}_{r\in [0,1]}$ be an analytic one-parameter family of smooth Lagrangian submanifolds having the same Liouville class of $\Lambda$, such that $\Lambda_0=\Lambda$ and $\Lambda_r \subset U_{\Sigma}^c$ for all $r\in [0,1]$. Then $\Lambda_r = \Lambda$ for all $r\in [0,1]$.
\end{mainthm}

The proof uses the characterization of strong chain transitivity in terms of Lyapunov functions from Theorem \ref{main1} (ii). 

\paragraph{Acknowledgements.} The first author is partially supported by the DFG grant AB 360/1-1. The second author is partially supported by the GNFM project ``Weak KAM Theory: dynamical aspects and applications''.

\section{Chain recurrence for flows in metric spaces}
\label{recsec}

Throughout the whole section, $\psi: \R \times X \rightarrow X$ denotes a continuous flow on a metric space $(X,d)$. We use the standard notation $\psi(t,x) = \psi_t(x)$. We begin by recalling some standard definitions.

\begin{defn}{\em (Chain recurrence)} 
\par\noindent
(i) Given $x,y\in X$, $\epsilon > 0$ and $T \geq 0$, a $(\epsilon,T)$-chain from $x$ to $y$  is a finite sequence $(x_i,t_i)_{i=1,\dots, n} \subset X \times \mathbb{R}$ such that  $t_i \geq T$ for all $i$,  $x_1 = x$ and setting $x_{n+1} := y$, we have
\[
d(\psi_{t_i}(x_i),x_{i+1}) < \epsilon \qquad \forall i = 1, \dots, n.
\]
\par\noindent
(ii) A point $x \in X$ is said to be chain recurrent if for all $\epsilon > 0$ and $T \geq 0$, there exists a $(\epsilon,T)$-chain from $x$ to $x$.  The set of chain recurrent points is denoted by $\mathcal{CR}(\psi)$.  
\par\noindent 
(iii) The flow $\psi$ on $X$ is said to be chain recurrent if $\mathcal{CR}(\psi)=X$.
\end{defn}

\begin{defn}{\em (Strong chain recurrence)} \label{scrdef}  \par\noindent
(i) Given $x,y\in X$, $\epsilon > 0$ and $T \geq 0$, a strong $(\epsilon,T)$-chain from $x$ to $y$ is a finite sequence $(x_i,t_i)_{i=1,\dots, n} \subset X \times \mathbb{R}$ such that $t_i \geq T$ for all $i$,  $x_1 = x$ and setting $x_{n+1} := y$, we have
\[
\sum_{i=1}^n d(\psi_{t_i}(x_i),x_{i+1}) < \epsilon.
\]
(ii) A point $x \in X$ is said to be strongly chain recurrent if for all $\epsilon > 0$ and $T \geq 0$, there exists a strong $(\epsilon,T)$-chain from $x$ to $x$. The set of strongly chain recurrent points is denoted by $\mathcal{SCR}(\psi)$.  
\par\noindent 
(iii) The flow $\psi$ on $X$ is said to be strongly chain recurrent if $\mathcal{SCR}(\psi)=X$.
\end{defn}

Clearly, the set of fixed points, and more generally periodic points, is contained in the strongly chain recurrent set $\mathcal{SCR}(\psi)$, which is contained in the chain recurrent set $\mathcal{CR}(\psi)$. Both $\mathcal{CR}(\psi)$ and $\mathcal{SCR}(\psi)$ are easily seen to be invariant closed sets (see e.g.\ \cite{con88,aki93,zhe00}). 

Now we recall the notions of Lyapunov function and first integral: the latter notion is unambiguous, the former one has different definitions in the literature. Here we adopt the following one.

\begin{defn}{\em (Lyapunov function and neutral set)}
\par\noindent
A function $h:X\rightarrow \R$ is a Lyapunov function for $\psi$ if $h\circ \psi_t \leq h$ for every $t\geq 0$. The neutral set $\mathcal{N}(h)$ of a Lyapunov function $h$ is the set of all points $x\in X$ such that the function $t\mapsto h\circ \psi_t (x)$ is constant.
\end{defn}

By the group property of $\psi$, the function $h:X\rightarrow \R$ is Lyapunov if and only if for every $x\in X$ the function $t\mapsto h\circ \psi_t(x)$ is monotonically decreasing on $\R$. The neutral set of a Lyapunov function $h$ is invariant, and it is closed if $h$ is continuous.

\begin{defn}{\em (First integral)}
\par\noindent
A function $h:X\rightarrow \R$ is a first integral for $\psi$ if $h\circ \psi_t = h$ for every $t\in \R$.
\end{defn}

In other words, $h$ is a first integral if and only if $h$ is a Lyapunov function with $\mathcal{N}(h)=X$. The following lemma is useful in order to characterize Lyapunov functions and first integrals in the case of Lipschitz regularity.

\begin{lem} 
\label{lipfirstint}
Let $M$ be a manifold and $V$ a locally Lipschitz continuous vector field on $M$, inducing a complete flow $\psi:\R \times M \rightarrow M$. Let $h:M \rightarrow \R$ be a locally Lipschitz continuous function. Then:
\begin{enumerate}[(i)]
\item $h$ is a Lyapunov function for $\psi$ if and only if $dh\circ V\leq 0$ almost everywhere.
\item $h$  is a first integral for $\psi$ if and only if $dh\circ V= 0$ almost everywhere.
\end{enumerate}
\end{lem}  

\begin{proof} (i)  Assume that $h$ is a Lyapunov function for $\psi$. Then we have
\[
dh\circ V(x) = \lim_{t\to 0}\frac{h\circ \psi_t (x)-h(x)}{t} \leq 0
\]
for every $x\in M$ which is a differentiability point for $h$. Being locally Lipschitz continuous, $h$ is almost everywhere differentiable, and hence $dh\circ V\leq 0$ almost everywhere.

Now we suppose that $dh\circ V\leq 0$ almost everywhere and assume by contradiction that $h$ is not a Lyapunov function. This means that there are $x\in X$ and $t>0$ such that
\[
a:= h(\psi_t(x)) - h(x) >0.
\]
Since $h$ and $\psi_t$ are locally Lipschitz continuous, we can find $r>0$ and $c\geq 0$ such that
\[
h(\psi_t(y)) - h(y) \geq a - c \, d(y,x) \qquad \forall y\in B_r(x),
\]
where $d$ is a distance function on $M$ induced by some Riemannian metric and $B_r(x)$ denotes the corresponding ball of radius $r$ centered at $x$. By integrating this inequality on a ball of radius $\epsilon\leq r$ we find
\[
\int_{B_{\epsilon}(x)} \bigl( h(\psi_t(y)) - h(y)\bigr) \, d\mu(y) \geq a \,\mu\bigl(B_{\epsilon}(x)\bigr) - c \int_{B_{\epsilon}(x)} d(y,x)\, d\mu(y) \geq (a-c\epsilon) \mu\bigl(B_{\epsilon}(x)\bigr),
\]
where $\mu$ denotes the measure on $M$ which is induced by the Riemannian metric. So for $\epsilon<a/c$ we have
\begin{equation}
\label{ass}
\int_{B_{\epsilon}(x)} \bigl( h(\psi_t(y)) - h(y)\bigr) \, d\mu(y) >0.
\end{equation}
On the other hand, using the fact that the function $s\mapsto h(\psi_s(y))$ is locally Lipschitz continuous and hence a.e.\ differentiable, Tonelli's theorem, the a.e.\ differentiability of $h$ and the chain rule for Lipschitz continuous maps, we obtain
\[
\begin{split}
\int_{B_{\epsilon}(x)} \bigl( h(\psi_t(y)) - h(y)\bigr) \, d\mu(y) &= \int_{B_{\epsilon}(x)} \left( \int_0^t \frac{d}{ds} h(\psi_s(y))\, ds \right) \, d\mu(y) \\ &= \int_0^t \left( \int_{B_{\epsilon}(x)}  \frac{d}{ds} h(\psi_s(y))\, d\mu(y) \right) \, ds \\ &= \int_0^t \left( \int_{B_{\epsilon}(x)}  dh(\psi_s(y)) [V(\psi_s(y))]\, d\mu(y) \right) \, ds.
\end{split}
\]
Since $dh\circ V\leq 0$ a.e., the latter quantity is non-positive, and this contradicts (\ref{ass}). This proves that $h$ is a Lyapunov function.

\noindent (ii) By (i), the fact that $dh\circ V$ vanishes almost everywhere is equivalent to the fact that both $h$ and $-h$ are Lyapunov functions, and hence to the fact that $h$ is a first integral.
\end{proof}

Here is a proposition which says that in a strongly chain recurrent dynamics the only Lipschitz continuous Lyapunov functions are first integrals.

\begin{prop}\label{SCR} 
The neutral set of a Lipschitz continuous Lyapunov  function for the flow $\psi$ contains the strongly chain recurrent set of $\psi$. In particular, if $\psi$ is strongly chain recurrent then every Lipschitz continuous Lyapunov function is a first integral.
\end{prop}

\begin{proof}
Let $h$ be a Lipschitz continuous Lyapunov function for $\psi$. Assume by contradiction that an orbit in the strongly chain recurrent set of $\psi$ is not contained in the neutral set of $h$. Therefore, there exist $x\in \mathcal{SCR}(\psi)$ and $T\geq 0$ such that
\[
h\bigl( \psi_{T}(x) \bigr) < h(x).
\]
We choose a positive number $\epsilon$ such that
\[
\epsilon < \frac{h(x)- h\bigl( \psi_{T}(x) \bigr)}{2c},
\]
where $c$ is a Lipschitz constant for $h$.  Let $(x_i,t_i)_{i=1,\dots,n}$ be a strong $(\epsilon,T)$-chain from $x$ to $x$. Since $t_1\geq T$ we have
\[
h\bigl( \psi_{t_1}(x_1) \bigr) = h\bigl( \psi_{t_1}(x) \bigr) \leq h \bigl( \psi_T(x) \bigr).
\]
Therefore, using the fact that $h$ is $c$-Lipschitz and the property of $(\epsilon,T)$-chain, we find
\begin{equation}
\label{uno}
h(x_2) - h \bigl( \psi_T(x) \bigr) \leq h(x_2) - h \bigl( \psi_{t_1}(x_1) \bigr) \leq c \, d\bigl(x_2,\psi_{t_1}(x_1) \bigr) < c \epsilon.
\end{equation}
Moreover, we have for every $i=1,\dots,n$, 
\[
h(x_{i+1}) - h(x_i) \leq h(x_{i+1}) - h \bigl( \psi_{t_i}(x_i) \bigr) \leq c \, d\bigl(x_{i+1},\psi_{t_i}(x_i) \bigr),
\]
where we have set $x_{n+1}:= x$. By adding these inequalities for $i=2,\dots,n$ and using the property of strong $(\epsilon,T)$-chain, we obtain
\begin{equation}
\label{due}
h(x) - h(x_2) \leq c \sum_{i=2}^n d \bigl( x_{i+1}, \psi_{t_i}(x_i) \bigr) < c \epsilon.
\end{equation}
From (\ref{uno}) and (\ref{due}) we deduce that
\[
h(x) - h \bigl( \psi_T(x) \bigr) < 2 c \epsilon,
\]
which contradicts the choice of $\epsilon$.
\end{proof}

The next two examples show that Proposition \ref{SCR} is somewhat optimal: one cannot replace the strongly chain recurrent set by the chain recurrent set,  and one cannot ask the Lyapunov function to be only continuous. Indeed, we shall construct two smooth flows on the one-torus $\T:= \R/\Z$ having the following properties:
\begin{enumerate}[(i)]
\item The first flow is chain recurrent (but not strongly) and admits a Lipschitz continuous Lyapunov function which is not a first integral.
\item The second  flow is strongly chain recurrent and admits a continuous (but not Lipschitz continuous) Lyapunov function which is not a first integral.
\end{enumerate}
These examples are continuous-time versions of the discrete-time examples appearing in \cite{fp15}[Example 2.2].

\begin{ex} \label{cantor} \textnormal{Endow $\T=\R/\Z$ with the standard quotient metric, and consider a Cantor set $K \subset \T$ of positive Lebesgue measure $\mu(K) = \delta > 0$. Let $\varphi: \T \to [0,+\infty)$ be a non negative smooth function whose set of zeroes is $K$. Let $\psi:\R \times \T \rightarrow \T$ be the flow of the vector field
\[
V(x) := \varphi(x) \frac{\partial}{\partial x}.
\]
Every point $x \in \T$ is chain recurrent, that is $\mathcal{CR}(\psi) = \T$. Indeed, points in $K$ are chain recurrent, being fixed points. The complement of $K$ consists of open intervals, on which the flow moves points in the same direction. Since $K$ consists of fixed points, a $(\epsilon,T)$-chain is allowed to contain subchains of the form $(x_i,T)_{i=h,\dots,k}$, where $x_i$ are points in $K$ with $d(x_i,x_{i+1})<\epsilon$. Therefore, it is always possible to return to $x$ with a $(\epsilon,T)$-chain, by using the flow $\psi$ and an appropriate sequence of jumps on $K$ (everyone of amplitude smaller than $\epsilon$).
\\
Denoting by $\mathbbm{1}_{K}$ and $\mathbbm{1}_{K^c}$ the characteristic functions of the Cantor set $K$ and its complement $K^c$, respectively, we consider the function
\[
h(x) := \frac{1}{\delta} \int_{0}^x \mathbbm{1}_{K}(t)\, dt - \frac{1}{1 - \delta} \int^{x}_0 \mathbbm{1}_{K^c}(t) \, dt \qquad  \forall x\in [0,1].
\]
Since $h(1)=0=h(0)$, $h$ can be seen as a function on $\T$. We easily check that $h$ is a Lipschitz continuous Lyapunov function for $\psi$, which is not a first integral. Indeed,
\[
|h(x)-h(y)| \le 
 \left( \frac{1}{\delta}+\frac{1}{1-\delta} \right)
 |x-y| \qquad \forall x,y \in \T.
\]
Moreover, $h$ is almost everywhere differentiable and 
\[
\begin{split}
dh(x)[V(x)] &= h'(x) \varphi(x) = \left(\frac{1}{\delta} \mathbbm{1}_K(x) - \frac{1}{1-\delta} \mathbbm{1}_{K^c}(x) \right) \varphi(x) \\ &= - \frac{1}{1-\delta} \mathbbm{1}_{K^c}(x) \varphi(x) \leq 0 \qquad \mbox{for a.e. } x\in \T, 
\end{split}
\]
so $h$ is a Lyapunov function by Lemma \ref{lipfirstint}. It is not a first integral, being strictly decreasing along the orbit of each point in the complement of $K$.
}
\end{ex}

\begin{ex} \label{cantor1} 
\textnormal{We define a smooth flow $\psi$ on $\T$ by the same construction of the previous example, but starting from a Cantor set $K$ of zero Lebesgue measure. The fact that $K$ has zero measure now implies that every $x\in K$ is strongly chain recurrent. Indeed, the zero-measure set $K$ can be overstepped by alternating flow lines and jumps whose total amplitude is smaller that $\epsilon$.
\\
In order to construct a continuous Lyapunov function for $\psi$, we see $K$ as the intersection $K=\bigcap_n K_n$, where each $K_n$ is a finite union of closed intervals of total measure $\delta_n$, with $\delta_n\rightarrow 0$. For every $n$, consider the function on $\T$ 
\[
h_n(x) := \frac{1}{\delta_n} \int_{0}^x \mathbbm{1}_{K_n}(t) \, dt - \frac{1}{1 - \delta_n} \int^{x}_0 \mathbbm{1}_{K_n^c}(t)\,  dt,
\]
which is increasing on each open interval that is a connected component of $K_n$ and decreasing on each open interval that is a connected component of $K_n^c$. Moreover, since $(h_n)_{n \in \mathbb{N}}$ is a Cauchy sequence in the uniform norm, it converges uniformly to a continuous function $h$ on $\T$. By the very construction of $\psi$ and $h_n$, the function $h$ is a continuous Lyapunov function which is not a first integral.}
\end{ex}

If $\psi$ is the chain recurrent flow which is described in Example \ref{cantor} or \ref{cantor1}, then every $C^1$ Lyapunov function $h$ for $\psi$ is a first integral. Indeed, in this case the inequality
\[
0 \geq dh(x)[V(x)] = h'(x) \varphi(x) \qquad \forall x\in \T,
\]
implies that $h'(x)\leq 0$, so $h$ is monotone and, being periodic, must be constant. 

In dimension greater than one, it is possible to construct chain recurrent smooth flows admitting a $C^1$ Lyapunov function which is not a first integral. The possibility of having these examples is related to the failure of Sard's theorem for real valued functions which are just $C^{k-1}$ on a $k$-dimensional manifold and was already observed by Hurley in \cite{hur95}. The first complete construction we are aware of is due to Pageault, who in \cite[Theorem 5.3]{pag13} shows that on every connected closed manifold of dimension at least two one can find a smooth flow for which every point is chain recurrent and which admits a $C^1$ Lyapunov function which is not a first integral.

The key step in Pageault's construction is to build a $C^1$ function $h$ whose set of critical points is homeomorphic to an interval. Such an example sharpens the classical example by Whitney, who in \cite{whi35} constructs a $C^1$ function which is not constant on an arc of critical points, but which might have more critical points. Then one considers the flow which is induced by a smooth pseudo-gradient vector field for $h$, i.e. a smooth vector field $V$ whose set of singular points coincides with the critical set of $h$ and such that $dh\circ V<0$ on the complement of this set.

\section{Characterization of strong chain recurrence}
\label{carsec}

Proposition \ref{SCR} says that a point which is not in the neutral set of a Lipschitz continuous Lyapunov function is not strongly chain recurrent. This result  has the following partial converse:

\begin{prop}
\label{pierrecont}
Let $\psi$ be a flow on the metric space $X$. If $x\in X$ is not strongly chain recurrent, then there exists a continuous Lyapunov function $h: X\to \R$ such that $x$ does not belong to the neutral set of $h$. If moreover $\psi_t$ is Lipschitz continuous for every $t\geq 0$, uniformly for $t$ on compact subsets of $[0,+\infty)$, then $h$ can be chosen to be Lipschitz continuous.
\end{prop}

This proposition, together with Proposition \ref{SCR}, immediately implies that the strongly chain recurrent set of a flow has the following characterization in terms of Lipschitz continuous Lyapunov functions.

\begin{thm}
\label{pierrethm}
Let  $\psi$ be a flow on the metric space $X$ such that $\psi_t$ is Lipschitz continuous for every $t\geq 0$, uniformly for $t$ on compact subsets of $[0,+\infty)$. Then
\[
\mathcal{SCR}(\psi) = \bigcap_{h} \mathcal{N}(h),
\]
where the intersection is taken over all Lipschitz continuous Lyapunov functions for $\psi$. In particular, $\psi$ is strongly chain recurrent if and only if  every Lipschitz continuous Lyapunov function is a first integral.
\end{thm}

This proves statement (i) in Theorem \ref{main1} from the Introduction.

\begin{rem}
\label{remconley}
{\em The above result should be compared to a theorem of Conley's, which states that a continuous flow on a compact metric space admits a continuous Lyapunov function whose neutral set coincides with the chain recurrent set, see \cite[Section II.6.4]{con78}. On the one hand the above theorem is stronger since it produces Lipschitz continuous Lyapunov functions whose neutral set does not contain a given orbit in the complement of  strongly chain recurrent set (which is in general strictly larger than the complement of the chain recurrent set). On the other hand, we do not obtain a single Lyapunov function whose neutral set coincides with the strongly chain recurrent set. This stronger version of Conley's theorem has been proved by Fathi and Pageault in the framework of discrete-time dynamical systems, see \cite{fp15} and \cite{pag11}[Section 2.4]. Our proof uses their techniques.} 
\end{rem}

In order to prove Proposition \ref{pierrecont} we need to introduce some objects. Let $x\in X$. For any $T\geq 0$ we indicate by $\mathcal{C}_T(x)$ the set of chains $(x_i,t_i)_{i=1,\dots,n}$, $n\in \N$, such that $x_1=x$ and $t_i\geq T$ for every $i=1,\dots,n$.  If $y\in X$, the cost of going from $x$ to $y$ through the chain $u=(x_i,t_i)_{i=1,\dots,n} \in \mathcal{C}_T(x)$ is the non-negative quantity
\[
\ell(u,y) := \sum_{i=1}^n d(\psi_{t_i}(x_i),x_{i+1}),
\]
where we set $x_{n+1}:= y$ (see Definition \ref{scrdef}). Then we define a function 
\[
L_T: X \times X \rightarrow [0,+\infty) 
\]
as
\[
L_T(x,y) := \inf_{u \in \mathcal{C}_T(x)} \ell(u,y).
\]
Since the set $\mathcal{C}_T(x)$ is contained in $\mathcal{C}_{T'}(x)$ when $T\geq T'$, the function $T \mapsto L_T(x,y)$ is monotonically increasing for any pair $(x,y)\in X\times X$ and we set
\[
L_{\infty}(x,y) := \lim_{T\rightarrow +\infty} L_T(x,y) = \sup_{T\geq 0} L_T(x,y).
\]
We notice that 
\[
\mathcal{SCR}(\psi) = \{ x\in X \mid L_{\infty}(x,x)=0 \}.
\]
The following lemma summarizes the main properties of the function $L_T$. 

\begin{lem}
\label{propL_T}
For every $T\geq 0$ and every $x,y,z$ in $X$ there holds:
\begin{enumerate}[(i)]
\item $L_T(x,y) \leq L_T(x,z) + L_T(z,y)$;
\item $L_T(x,\psi_t(x))=0$ for every $t\geq T$;
\item $|L_T(x,y) - L_T(x,z)| \leq d(y,z)$.
\end{enumerate}
\end{lem}

\begin{proof} $(i)$ A chain in $\mathcal{C}_T(x)$ going from $x$ to $z$ and a chain in $\mathcal{C}_T(z)$ going from $z$ to $y$ can be concatenated to obtain a chain in  $\mathcal{C}_T(x)$ connecting $x$ to $y$. The triangular inequality is a straightforward consequence of this fact. \\
$(ii)$ For $t \ge T$ it is sufficient to consider the chain $\{ (x,t) \}$ from $x$ to $\psi_t(x)$.  \\
$(iii)$ For a fixed chain $u=(x_i,t_i)_{i=1,\dots,n} \in \mathcal{C}_T(x)$, we have 
\[
\ell(u,y) \le \ell(u,z) + d(y,z).
\]
By taking the infimum over all chains in $\mathcal{C}_T(x)$, we obtain
\[
L_T(x,y) \le L_T(x,z) + d(y,z).
\]
Exchanging the role of $y$ and $z$, the desired inequality immediately follows.
\end{proof}

\begin{proof}[Proof of Proposition \ref{pierrecont}]
Let $x\in X$ be a point which is not strongly chain recurrent for $\psi$, that is, $L_{\infty}(x,x)>0$. Choose $T\geq 0$ large enough so that $L_T(x,x)>0$ and define
\[
\tilde{h}(y):= L_T(x,y) \qquad \forall y\in X.
\]
The function $\tilde{h}$ is 1-Lipschitz by statement (iii) in Lemma \ref{propL_T}. Moreover, statements (i) and (ii) in the same lemma imply that for every $y\in X$ we have
\begin{equation}
\label{pierre1}
\tilde{h}(\psi_t(y)) = L_T(x,\psi_t(y)) \leq L_T(x,y) + L_T(y,\psi_t(y)) =   L_T(x,y) = \tilde{h}(y) \qquad \forall t\geq T.
\end{equation}
Furthermore,
\begin{equation}
\label{pierre2}
\tilde{h}(\psi_t(x)) = L_T(x,\psi_t(x)) = 0 < L_T(x,x) = \tilde{h}(x) \qquad \forall t\geq T.
\end{equation}
In general, the function $\tilde{h}$ is not a Lyapunov function for $\psi$ because (\ref{pierre1}) holds only for $t\geq T$. We can obtain a Lyapunov function from $\tilde{h}$ by setting
\[
h(y):= \max_{s\in [0,T]} \tilde{h}(\psi_s(y)) \qquad \forall y\in X.
\]
Indeed, given $y\in X$, (\ref{pierre1}) implies 
\[
\tilde{h}(\psi_r(y)) \leq \max_{s\in [0,T]} \tilde{h}(\psi_s(y)) = h(y) \qquad \forall r\geq 0,
\]
and, by taking the maximum for $r\in [t,t+T]$, we obtain
\[
h(\psi_t(y))  \leq h(y) \qquad \forall t\geq 0.
\]
Hence $h$ is a Lyapunov function for $\psi$. Moreover, (\ref{pierre2}) implies
\[
\tilde{h}(\psi_{T+s}(x)) < \tilde{h}(x) \qquad \forall s\geq 0,
\]
and by taking the maximum for $s\in [0,T]$ we find
\[
h(\psi_T(x)) < \tilde{h}(x) \leq h(x).
\]
This shows that $x$ is not in the neutral set of the Lyapunov function $h$. 

There remains to show that $h$ is continuous, and that it is Lipschitz continuous if $\psi_t$ is Lipschitz continuous for every $t\geq 0$, uniformly for $t$ on compact subsets of $[0,+\infty)$. Let $y,z\in X$ and let $s^*\in [0,T]$ be such that
\[
h(y) = \tilde{h}(\psi_{s^*}(y)).
\]
From
\[
h(z) \geq \tilde{h}(\psi_{s^*}(z))
\]
we obtain, using the fact that $\tilde{h}$ is 1-Lipschitz
\[
h(y) - h(z) \leq \tilde{h}(\psi_{s^*}(y)) - \tilde{h}(\psi_{s^*}(z)) \leq d(\psi_{s^*}(y),\psi_{s^*}(z)).
\]
By exchanging the role of $y$ and $z$ we find
\[
|h(y) - h(z)| \leq \max_{s\in [0,T]} d(\psi_{s}(y),\psi_{s}(z))
\]
This inequality shows that $h$ is in general continuous, and Lipschitz continuous whenever $\psi_t$ is Lipschitz continuous uniformly for $t\in [0,T]$.
\end{proof}

\section{Chain transitivity for flows in metric spaces}

We now recall the definition of chain transitive and strongly chain transitive flow on a metric space. 

\begin{defn}
The flow $\psi$ on the metric space $(X,d)$ is said to be (strongly) chain transitive if for every $x,y\in X$, every $\epsilon>0$ and every $T\geq 0$ there is a (strong) $(\epsilon,T)$-chain from $x$ to $y$.
\end{defn}

(Strongly) chain transitive flows are a fortiori (strongly) chain recurrent. The flows of Examples \ref{cantor} and \ref{cantor1} are chain transitive, but only that of Example \ref{cantor1} is strongly chain transitive. 

We have seen in the Proposition \ref{SCR} that strong chain recurrence forces all Lipschitz continuous Lyapunov function to be first integral. Similarly, strong chain transitivity implies all Lipschitz continuous Lyapunov function to be constant.

\begin{prop}
\label{ppp1}
If $\psi$ is strongly chain transitive then every Lipschitz continuous Lyapunov function is constant.
\end{prop}

\begin{proof}
Assume that $h: X \rightarrow \R$ is a $c$-Lipschitz continuous Lyapunov function for $\psi$ and fix two points $x,y\in X$. For every $\epsilon>0$ we can find a strong $(\epsilon,0)$-chain, that is a finite sequence $(x_i,t_i)_{i=1,\dots,n}$ such that $x_1=x$, $t_i\geq 0$ and, setting $x_{n+1}:= y$,
\[
\sum_{i=1}^n d(\psi_{t_i}(x_i),x_{i+1}) < \epsilon.
\]
From the fact that $h$ is a Lyapunov function and is $c$-Lipschitz continuous, we find for every $i=1,\dots n$
\[
h(x_{i+1}) - h(x_i) \leq h(x_{i+1}) - h(\psi_{t_i}(x_i)) \leq c \, d\bigl( \psi_{t_i}(x_i), x_{i+1} \bigr),
\]
and summing over all such indexes $i$ we get
\[
h(y) - h(x) \leq c \sum_{i=1}^n d(\psi_{t_i}(x_i),x_{i+1}) < c \epsilon.
\]
Therefore, $h(y)\leq h(x)$. By exchanging the role of $x$ and $y$ we conclude that $h(y)=h(x)$, and hence $h$ is constant.
\end{proof}

Here is a partial converse of the above proposition.

\begin{prop}
\label{ppp2}
Let $\psi$ be a flow on the metric space $(X,d)$. Assume that every continuous Lyapunov function for $\psi$ is constant. Then $\psi$ is strongly chain recurrent. If moreover $\psi_t$ is Lipschitz continuous  for every $t\geq 0$, uniformly for $t$ on compact subsets of $[0,+\infty)$, then the above conclusion holds also assuming only that every Lipschitz continuous Lyapunov function for $\psi$ is constant.
\end{prop}

\begin{proof}
Assume by contradiction that $\psi$ is not strongly chain recurrent. Our aim is to construct a non-constant continuous Lyapunov function for $\psi$, which is Lipschitz continuous under the further assumption that $\psi_t$ is Lipschitz continuous  for every $t\geq 0$, uniformly for $t$ on compact subsets of $[0,+\infty)$.

Since $\psi$ is not strongly chain recurrent,
there exist $x,y\in X$, $\epsilon>0$ and $T\geq 0$ such that for every $(x_i,t_i)_{i=1,\dots,n} \subset X \times [T,+\infty)$ with $x_1=x$ there holds
\[
\sum_{i=1}^n d(\psi_{t_i}(x_i),x_{i+1}) \geq \epsilon,
\]
where we have set $x_{n+1}:= y$.  By using the notation introduced in Section \ref{carsec}, this means that
\[
L_T(x,y) \geq \epsilon.
\]
The function
\[
\tilde{h}: X \rightarrow \R, \qquad \tilde{h}(z) := L_T(x,z)
\]
is 1-Lipschitz by statement (iii) in Lemma \ref{propL_T}. Moreover,
\begin{equation}
\label{jj1}
\tilde{h}(y) = L_T(x,y) \geq \epsilon,
\end{equation}
while
\begin{equation}
\label{jj2}
\tilde{h}(\psi_t(x)) = L_T(x,\psi_t(x)) = 0 \qquad \forall t\geq T,
\end{equation}
by statement (ii) in Lemma \ref{propL_T}. By statements (i) and (ii) in Lemma \ref{propL_T} we also have for every $z\in X$ and $t\geq T$
\[
\tilde{h}(\psi_t(z)) = L_T(x,\psi_t(z)) \leq L_T(x,z) + L_T(z,\psi_t(z)) = L_T(x,z) = \tilde{h}(z).
\]
Arguing as in the proof of  Proposition \ref{pierrecont}, we see that
\[
h(z) := \max_{s\in [0,T]} \tilde{h}(\psi_s(z))
\]
is a continuous Lyapunov function for $\psi$, and it is Lipschitz continuous when  $\psi_t$ is Lipschitz continuos uniformly for $t\in [0,T]$. There remains to check that $h$ is not constant. By (\ref{jj1}) we have
\[
h(y) \geq \tilde{h}(y) \geq \epsilon >0,
\]
while by (\ref{jj2}) we get
\[
h(\psi_T(x)) = \max_{t\in [T,2T]} \tilde{h}(\psi_t(x)) = 0,
\]
concluding the proof.
\end{proof}

Propositions \ref{ppp1} and \ref{ppp2} have the following immediate consequence, which is statement (ii) in Theorem \ref{main1} from the introduction.

\begin{cor}
\label{coro}
Let $\psi$ be a flow on the metric space $(X,d)$ such that $\psi_t$ is Lipschitz continuous  for every $t\geq 0$, uniformly for $t$ on compact subsets of $[0,+\infty)$. Then $\psi$ is strongly chain transitive if and only if every Lipschitz continuous Lyapunov function is constant. 
\end{cor}

\begin{rem}
\textnormal{Consider the following weaker version of strong chain transitivity: for every $x,y\in X$ and every $\epsilon>0$ there exists a strong $(\epsilon,0)$-chain from $x$ to $y$. In other words, we do not require the $t_i$'s in the chain $(x_i,t_i)_{i=1,\dots,n}$ to be larger than an arbitrarily given number $T\geq 0$, but just non-negative. As the proof of Proposition \ref{ppp1} shows, this property implies that every Lipschitz continuous Lyapunov function is constant. Conversely, if every Lipschitz continuous Lyapunov function is constant then this property holds: indeed, if there are $x,y\in X$ and $\epsilon>0$ such that there is no strong $(\epsilon,0)$-chain from $x$ to $y$, then $L_0(x,y)\geq \epsilon$, and $h(z):= L_0(x,z)$ is a Lipschitz continuous Lyapunov function which is not constant because
\[
h(x) = L_0(x,x) = 0 < \epsilon \leq L_0(x,y) = h(y).
\]
Therefore, the above property is equivalent to the fact that every Lipschitz continuous Lyapunov function is constant. Under the additional assumption that $\psi_t$ is Lipschitz continuous  for every $t\geq 0$, uniformly for $t$ on compact subsets of $[0,+\infty)$, Corollary \ref{coro} implies that this property is equivalent to strong chain transitivity. For an arbitrary flow we suspect this property to be strictly weaker than strong chain transitivity.}
\end{rem}

It is well known that ergodicity with respect to a measure which is positive on every non-empty open set implies chain transitivity. 
We conclude this section by showing that it actually implies strong chain transitivity. See \cite{zhe98} for related results in the case of homeomorphisms. Here we assume the metric space $(X,d)$ to be compact. Let $\mu$ be a probability measure on the Borel $\sigma$-algebra of $X$. We recall that a continuous flow $\psi$ on $X$ preserving the measure $\mu$ is {\em ergodic} with respect to $\mu$ if and only if for every $f\in C(X)$ and $g\in L^1(\mu)$ there holds
\begin{equation}
\label{ergo}
\lim_{T\rightarrow +\infty} \frac{1}{T} \int_0^T \left( \int_X f(\psi_t(x)) g(x)\, d\mu(x)  \right) \, dt = \int_X f\, d\mu  \int_X g\, d\mu.
\end{equation}
See e.g. \cite[Lemma 6.11]{wal82}.   

\begin{prop} \label{ergodic} Suppose that the flow $\psi$ on the compact metric space $(X,d)$ is ergodic with respect to a Borel measure $\mu$ which is positive on every non-empty open set. Then every continuous Lyapunov function for $\psi$ is  constant. In particular, $\psi$ is strongly chain transitive.
\end{prop}

\begin{proof}
Let $h$ be a continuous Lyapunov function for $\psi$.
Set $a:=\min h$ and consider for $\epsilon>0$ the non-empty open set 
\[
A_{\epsilon}:= \{x\in X \mid h(x)< a+ \epsilon\}.
\]
Since $A_{\epsilon}$ is positively invariant, we have the bounds
\[
a \, \mu(A_{\epsilon}) \leq \int_X h(\psi_t(x)) \mathbbm{1}_{A_{\epsilon}} (x) \, d\mu(x) \leq (a+\epsilon) \, \mu(A_{\epsilon}),
\]
for every $t\geq 0$. Since $\psi$ is ergodic, by inserting $f=h$ and $g=\mathbbm{1}_{A_{\epsilon}}$ in (\ref{ergo}) we find that the number
\[
\mu(A_{\epsilon}) \int_X h\, d\mu = \lim_{T\rightarrow +\infty} \frac{1}{T} \int_0^T \left( \int_X h(\psi_t(x)) \mathbbm{1}_{A_{\epsilon}} (x) \, d\mu(x)   \right) \, dt
\]
belongs to the interval $[a \mu(A_{\epsilon}),(a+\epsilon) \mu(A_{\epsilon})]$. Since $\mu(A_{\epsilon})>0$, we deduce that
\[
\int_X h\, d\mu \in [a,a+\epsilon]
\]
Since $\epsilon$ is arbitrary, we obtain that
\[ 
\int_X h\, d\mu = a.
\]
Together with the fact that $h\geq a$, this implies that $h=a$ $\mu$-a.e. Since $h$ is a continuous function and $\mu$ is positive on non-empty open sets, we deduce that $h=a$ everywhere. The last assertion follows from Proposition \ref{ppp2}.
\end{proof}

\section{Lagrangian submanifolds of cotangent bundles and their Liouville class}
\label{liouvillesec}

In the remaining part of the paper, we consider the cotangent bundle $T^*M$ of a closed manifold $M$. Points in $T^*M$ are denoted as $(x,y)$, with $x\in M$ and $y\in T_x^*M$. We denote by $\lambda$ the Liouville form of $T^*M$, that is the one-form whose expression in local cotangent coordinates is
\[
\lambda(x^1,\dots,x^n,y_1,\dots,y_n) = \sum_{j=1}^n y_j \, dx^j.
\]
We equip $T^*M$ with the symplectic form $\omega=d\lambda$. 

We denote by $\mathcal{L}(T^*M)$ the set of closed Lagrangian submanifolds $\Lambda$ of $T^*M$ which are Lagrangian isotopic to the zero-section $\mathcal{O}$ of $T^*M$. This means that there is a smooth family of Lagrangian embeddings
\[
\varphi_t : M \rightarrow T^*M, \qquad t\in [0,1],
\]
such that $\Lambda=\varphi_1(M)$ and $\varphi_0$ is the standard embedding onto the zero-section  $\mathcal{O}$. Let $\Lambda$ be in $\mathcal{L}(T^*M)$. The Liouville class of $\Lambda$ is an element of the first De Rham cohomology group $H^1(M,\R)$ and is defined as follows. Choose a smooth family of Lagrangian embeddings $\varphi_t$ as above and denote by $\tilde{\varphi}_1: M \rightarrow \Lambda$ the map which is obtained from $\varphi_1$ by restriction of the codomain. Denote by $\imath_{\Lambda} : \Lambda \hookrightarrow T^*M$ the inclusion. Since $\Lambda$ is Lagrangian, the one-form $\imath_{\Lambda}^* \lambda$ is closed, and hence defines an element $[\imath_{\Lambda}^* \lambda]$ of $H^1(\Lambda,\R)$. The {\em Liouville class} of $\Lambda$ is the cohomology class
\[
\mathrm{Liouville}(\Lambda) := \tilde{\varphi}_1^* [\imath_{\Lambda}^* \lambda] \in H^1(M,\R).
\]
Changing the family $\varphi_t$ produces a new map $\tilde{\varphi}_1$, which is homotopic to the previous one. This shows that the Liouville class of $\Lambda$ does not depend on the choice of the family $\varphi_t$. Lagrangian submanifolds with vanishing Liouville class are called {\em exact}.

Special elements of $\mathcal{L}(T^*M)$ are the images of closed forms: in fact, if $\theta$ is a one-form on $M$, then $\theta(M) \subset T^*M$ is a Lagrangian submanifold if and only of $\theta$ is closed. In the latter case, the isotopy $t\mapsto t\theta(M)$ shows that $\theta(M)$ is an element of $\mathcal{L}(T^*M)$. Here we are seeing a one form on $M$ as a section $\theta: M \rightarrow T^*M$. When $M$ is a torus one often sees a one form as a map $\theta: \T^n \rightarrow (\R^n)^*$ and consequently talks about the {\em graph} of $\theta$ in $T^* \T^n = \T^n \times (\R^n)^*$ rather than its image. The Liouville class of $\theta(M)$ is precisely the cohomology class of $\theta$:
\[
\mathrm{Liouville}(\theta(M)) = [\theta],
\]
for every closed one-form $\theta$.

Let $\theta$ be a one-form on $M$. The fiberwise translation 
\begin{equation}
\label{shift}
\sigma: T^*M \rightarrow T^*M, \qquad (x,y) \mapsto (x,y+\theta(x)),
\end{equation}
is a symplectic diffeomorphism if and only of $\theta$ is closed. In the next lemma we investigate the effect of a symplectic fiberwise translation on the Liouville class of an element of $\mathcal{L}(T^*M)$.

\begin{lem}
\label{Lio-shi}
Let $\Lambda\in \mathcal{L}(T^*M)$, let $\theta$ be a closed one-form on $M$, and let $\sigma$ be the symplectic fiberwise translation which is defined in (\ref{shift}). Then $\sigma(\Lambda)$ belongs to $\mathcal{L}(T^*M)$ and
\[
\mathrm{Liouville}(\sigma(\Lambda)) = \mathrm{Liouville}(\Lambda) + [\theta].
\]
\end{lem}

\begin{proof}
We can connect $\sigma$ to the identity by the smooth family of symplectic fiberwise translation
\[
\sigma_t : T^*M \rightarrow T^*M, \qquad (x,y) \mapsto (x,y+t\theta(x)), \qquad t\in [0,1].
\]
They satisfy
\begin{equation}
\label{eins}
\sigma_t^* \lambda = \lambda + t \pi^* \theta,
\end{equation}
where $\pi: T^*M \rightarrow M$ is the canonical projection. Let
\[
\varphi_t : M \rightarrow T^*M, \qquad t\in [0,1],
\]
be a smooth family of Lagrangian embeddings such that $\varphi_1(M)=\Lambda$ and $\varphi_0$ is the canonical embedding onto the zero-section. Then
\[
\sigma_t \circ \varphi_t : M \rightarrow T^*M, \qquad t\in [0,1],
\]
is a smooth family of Lagrangian embeddings such that $\sigma_1 \circ \varphi_1(M)=\sigma(\Lambda)$ and $\sigma_0 \circ \varphi_0=\varphi_0$ is the canonical embedding onto the zero-section. This shows that $\sigma(\Lambda)$ belongs to $\mathcal{L}(T^*M)$.

In order to compute the Liouville class of $\sigma(\Lambda)$, we set for simplicity $\varphi:= \varphi_1$ and we introduce the maps $\tilde{\varphi}$ and $\tilde{\sigma}$ by the following commutative diagram
\[
\xymatrix{ M \ar[r]^{\tilde{\varphi}} \ar[dr]^{\varphi} & \Lambda \ar[d]^{\imath_{\Lambda}} \ar[r]^{\tilde{\sigma}} & \sigma(\Lambda) \ar[d]^{\imath_{\sigma(\Lambda)}} \\ & T^*M \ar[r]^{\sigma} & T^*M}
\]
where the vertical maps $\imath_{\Lambda}$ and $\imath_{\sigma(\Lambda)}$ are inclusions. By using (\ref{eins}) for $t=1$ we compute:
\[
\begin{split}
\mathrm{Liouville}(\sigma(\Lambda)) &= (\tilde{\sigma} \circ \tilde{\varphi})^* [ \imath_{\sigma(\Lambda)}^* \lambda] = \tilde{\varphi}^* \tilde{\sigma}^* [ \imath_{\sigma(\Lambda)}^* \lambda] = \tilde{\varphi}^* [\tilde{\sigma}^*  \imath_{\sigma(\Lambda)}^* \lambda] \\
&= \tilde{\varphi}^* [\imath_{\Lambda}^* \sigma^* \lambda] = \tilde{\varphi}^* [\imath_{\Lambda}^* (\lambda + \pi^* \theta)] = \tilde{\varphi}^* [\imath_{\Lambda}^* \lambda] +  \tilde{\varphi}^*\imath_{\Lambda}^*\pi^* [\theta]\\
&= \mathrm{Liouville}(\Lambda) + \varphi^* \pi^* [\theta] = \mathrm{Liouville}(\Lambda) + \varphi_0^* \pi^* [\theta] \\ &= \mathrm{Liouville}(\Lambda) + [\theta].
\end{split}
\]
This concludes the proof.
\end{proof}

\section{Rigidity of Lagrangian submanifolds of optical hypersurfaces}
\label{rigiditysec}

A smooth Hamiltonian $H:T^*M\rightarrow \R$ induces the Hamiltonian vector field $X_H$ which is defined by $\imath_{X_H} \omega = - dH$. Its flow $\psi^H$ is symplectic and $H$ is a first integral. Any closed Lagrangian submanifold $\Lambda$ which is entirely contained in a level set of $H$ is automatically invariant for the Hamiltonian flow $\psi^H$.

Here we are interested in Tonelli Hamiltonians, that is, in smooth functions $H: T^*M \rightarrow \R$ which are fiberwise superlinear and whose second fiberwise differential is everywhere positive definite. Let $H$ be such a Tonelli Hamiltonian and let 
\[
\Sigma:= \{ z\in T^*M \mid H(z)=c \}
\]
be a non-empty regular level set of $H$ such that $\pi(\Sigma)=M$, where $\pi: T^*M \rightarrow M$ denotes the canonical projection. Hypersurfaces $\Sigma$ in $T^*M$ which can be obtained in this way are called {\em optical}. The hypersurface $\Sigma$ bounds the precompact open set 
\[
U_{\Sigma} := \{z \in T^*M \mid H(z) < c\}.
\]

The aim of this section is to reprove and discuss the following theorem of Paternain, Polterovich and Siburg, see \cite[Theorem 5.2]{pps03} and Theorem \ref{main2} in the Introduction.

\begin{thm}
\label{bound-rig}
Let $\Sigma$ be an optical hypersurface as above. Let $\Lambda$ be an element of $\mathcal{L}(T^*M)$ contained in $\Sigma$. Assume that  the restricted Hamiltonian flow
\[
\psi^H|_{\R\times \Lambda} : \R\times \Lambda \rightarrow \Lambda
\]
is strongly chain recurrent.
Then if $K$ is an element of $\mathcal{L}(T^*M)$ contained in $\overline{U_{\Sigma}}$ and having the same Liouville class of $\Lambda$, then necessarily $K=\Lambda$.
\end{thm}

In order to prove this theorem, we shall make use of {\em graph selectors}. These are  Lipschitz functions on $M$ whose properties are described in the following existence theorem:

\begin{thm} \label{GS} \textnormal{[Existence of graph selectors]} 
Assume that $\Lambda \in \mathcal{L}(T^*M)$ is exact. Then there exists a Lipschitz function $\Phi: M \to \mathbb{R}$ which is smooth on an open set $M_0 \subset M$ of full measure and such that 
\[
(x,d\Phi(x)) \in \Lambda,
\]
for every $x\in M_0$.
\end{thm}

We refer to \cite{cha91}, \cite{oh97} and \cite{sib04} for the proof of this theorem and for an extensive study of graph selectors. The next tool that we need is the following generalization of a theorem of Birkhoff, which is due to Arnaud.

\begin{thm}[\cite{arn10}] 
\label{arnaud}
Let $\Lambda$ be an element of $\mathcal{L}(T^*M)$. If $\Lambda$ is invariant with respect to the flow of a Tonelli Hamiltonian, then there exists a closed one-form $\theta$ on $M$ such that $\Lambda=\theta(M)$.
\end{thm}

Actually, in Arnaud's paper this theorem is proved for exact Lagrangians (in this case, the one-form $\theta$ is exact). The general case can be easily reduced to this special one by using a symplectic fiberwise translation, as we now explain.

\begin{proof}[Proof of Theorem \ref{arnaud}]
Let $a\in H^1(M,\R)$ be the Liouville class of $\Lambda$ and let $\theta$ be a closed one-form on $M$ whose cohomology class is $a$. Consider the symplectic translation
\[
\sigma: T^*M \rightarrow T^* M , \qquad (x,y) \mapsto (x,y-\theta(x)).
\]
By Lemma \ref{Lio-shi} the Lagrangian submanifold $\sigma(\Lambda)$ belongs to  $\mathcal{L}(T^*M)$ and has Liouville class
\[
\mathrm{Liouville}(\sigma(\Lambda)) = \mathrm{Liouville}(\Lambda) - [\theta] = a - [\theta] = 0.
\]
The Lagrangian submanifold $\sigma(\Lambda)$ is invariant with respect to the Hamiltonian flow of $H\circ \sigma^{-1}$, which is still Tonelli because $\sigma$ is a fiberwise translation. By the original version of Arnaud's theorem, there exists a smooth real function $u$ on $M$ such that 
$\sigma(\Lambda) = du(M)$. But then
\[
\begin{split}
\Lambda &= \sigma^{-1} (\sigma(\Lambda)) = \sigma^{-1} (du(M)) = \{(x,y+\theta(x)) \mid (x,y)\in du(M)\} \\
&= \{ (x,du(x) + \theta(x)) \mid x\in M\} = (du+\theta)(M).
\end{split}
\]
We conclude that $\Lambda$ is the image of the one-form $du+\theta$.
\end{proof}  

We can now prove Theorem \ref{bound-rig}.

\begin{proof}[Proof of Theorem \ref{bound-rig}] Let $H$ be a Tonelli Hamiltonian on $T^*M$ such that $\Sigma=H^{-1}(c)$. Since $\Lambda$ is contained in $H^{-1}(c)$, it is invariant with respect to the Hamiltonian flow of $H$. Then Theorem \ref{arnaud} implies the existence of a closed one-form $\theta$ on $M$ such that
\[
\Lambda = \theta(M).
\]
The symplectic fiberwise translation
\[
\sigma: T^*M \rightarrow T^* M , \qquad (x,y) \mapsto (x,y+\theta(x)),
\]
brings the zero-section $\mathcal{O}$ onto $\Lambda$. Moreover, $\sigma$ conjugates the Hamiltonian flow of $H$ with the Hamiltonian flow of
\[
\tilde{H} := H \circ \sigma - c,
\]
meaning that:
\[
\psi_t^H \circ \sigma = \sigma \circ \psi_t^{\tilde{H}} \qquad \forall t\in \R.
\]
The Hamiltonian $\tilde{H}$ is also Tonelli and vanishes on $\mathcal{O}$. Consider the vector field $Y$ on $M$ which is defined as
\[
Y(x) := d_y \tilde{H}(x,0)\in T_x M \qquad \forall x\in M,
\]
where $d_y$ denoted the fiberwise differential. Since the zero section $\mathcal{O}=\sigma^{-1}(\Lambda)$ is invariant with respect to $\psi^{\tilde{H}}$, the Hamiltonian vector field $X_{\tilde{H}}$ is tangent to $\mathcal{O}$ and coincides with $Y$ once $\mathcal{O}$ and $M$ are canonically identified. Moreover, a Taylor expansion at $(x,0)$ shows that we can write $\tilde{H}$ as
\[
\tilde{H}(x,y) = \langle y, Y(x) \rangle + F(x,y),
\]
where $\langle\cdot,\cdot \rangle$ denoted the duality pairing and the smooth function $F$ vanishes up to order one on the zero section. Being convex, $F|_{T_x^* M}$ achieves its minimum at the critical point $0$, where we have $F(x,0)=0$. We conclude that $F\geq 0$ on $T^*M$.  

Now let $K$ be an element of $\mathcal{L}(T^*M)$ contained in
\[
\overline{U_{\Sigma}} = \{ z\in T^*M \mid H(z) \leq c\}
\] 
and having the same Liouville class of $\Lambda$. The Lagrangian submanifold $\sigma^{-1}(K)$ is contained in 
\[
\{z\in T^*M \mid \tilde{H}(z) \leq 0\}
\]
and has Liouville class zero. The first fact implies that it admits a graph selector $\Phi:M \rightarrow \R$ as in Theorem \ref{GS}. The second one guarantees that
\begin{equation}
\label{sotto}
\tilde{H}(x,d\Phi(x)) = d\Phi(x)[Y(x)]  + F(x,d\Phi(x)) \leq 0 \qquad \mbox{for a.e. } x\in M.
\end{equation}
Since $F$ is non-negative, we have
\[
d\Phi\circ Y  \leq 0 \qquad \mbox{a.e.},
\]
and Lemma \ref{lipfirstint} implies that $\Phi$ is a Lyapunov function for the flow of $Y$. By identifying $M$ with the zero section $\mathcal{O}$, we obtain that $\Phi$ is a Lyapunov function for $\psi^{\tilde{H}}|_{\R\times \mathcal{O}}$, and hence that $\Phi\circ \sigma^{-1}$ is a Lipschitz Lyapunov function for $\psi^H|_{\R\times \Lambda}$. 

By Theorem \ref{pierrecont}, the hypothesis on the restriction of the flow $\psi^H$ to $\Lambda$ is equivalent to the fact that every Lipschitz continuous Lyapunov function for $\psi^H|_{\R \times \Lambda}$ is a first integral (here we just need the simpler implication which is given in Proposition \ref{SCR}). Therefore, $\Phi\circ \sigma^{-1}$ must be a first integral of $\psi^H|_{\R\times \Lambda}$, and hence $\Phi$ is a first integral of $\psi^{\tilde{H}}|_{\R\times \mathcal{O}}$. By Lemma \ref{lipfirstint} we deduce that
\[
d\Phi\circ Y  = 0 \qquad \mbox{a.e.},
\]
and (\ref{sotto}) implies that $F(x,d\Phi(x))\leq 0$ for almost every $x\in M$. By the properties of $F$, $d\Phi=0$ a.e.. Then $\sigma^{-1}(K)$ coincides with the zero section on a subset of full measure but, being a smooth submanifold, it coincides with the zero section everywhere: $\sigma^{-1}(K)=\mathcal{O}$. It follows that $K = \sigma(\mathcal{O}) = \Lambda$, as we wished to prove.
\end{proof}

The rigidity phenomenon for Lagrangian submanifolds of optical hypersurfaces does not hold true under the weaker hypothesis that all points of $\Lambda$ are chain recurrent for the restriction of the Hamiltonian flow. This fact has already been observed by Pageault in \cite[Remark 5.8]{pag13} (see also \cite{ffr09}) and follows from the fact that one can find gradient flows with a chain recurrent dynamics. Here is the construction:

\begin{ex} \textnormal{On a closed manifold $M$ of dimension at larger than $n$ one can find a $C^n$ function $h:M\rightarrow \R$ which is not constant and whose critical set is connected. An explicit example is constructed for $n=1$ in the already mentioned \cite{pag13}. If $h$ is such a function with $n\geq 3$, we define the Tonelli Hamiltonian of class $C^{n-1}$
\[
H: T^* M \rightarrow \R, \qquad H(x,y) = \frac{1}{2} \|y\|^2 - \langle y,\nabla h(x) \rangle,
\]
where the norm $\|\cdot\|$ on $T^*M$ and the gradient operator $\nabla$ are induced by some Riemannian metric on $M$.
The value $0$ is regular for $H$ and the zero section $\mathcal{O}$ is a Lagrangian submanifold of the optical hypersurface $\Sigma = H^{-1}(0)$. The restriction of the Hamiltonian flow of $H$ to $\mathcal{O}$ is the flow of the vector field $-\nabla h$, after the canonical identification of $\mathcal{O}$ with $M$. The fact that the critical set of $h$ is connected implies that every point is chain recurrent for the flow of $-\nabla h$, see \cite[Lemma 5.6]{pag13}. Therefore, every point of the exact Lagrangian submanifold $\mathcal{O}\subset \Sigma$ is chain recurrent for $\psi^H_{\R \times \mathcal{O}}$. Since
\[
H(x,dh(x)) = \frac{1}{2} \|dh(x)\|^2 - dh(x)[\nabla h] = - \frac{1}{2} \|dh(x)\|^2\leq 0 \qquad \forall x\in M,
\]
$\Lambda:= dh(M)\neq \mathcal{O}$ is a another exact Lagrangian submanifold (of class $C^{n-1}$) which is contained in $U_{\Sigma}$. 
\hfill $\Box$}
\end{ex}
\begin{rem}
\textnormal{Under the assumptions of Theorem \ref{bound-rig}, the energy value $c$ coincides with $\alpha(\eta)$, where $\eta\in H^1(M,\R)$ is the Liouville class of $\Lambda$ and $\alpha: H^1(M,\R) \rightarrow \R$ is Mather's $\alpha$-function. Indeed, this follows from the Hamiltonian characterization of the $\alpha$-function (see \cite{cipp98}[Theorem A] and also \cite{CIlibro}[Theorem 4-4.1], \cite{sor15}[Chapter 6]) thanks to the fact that $\Lambda$ is the image of a closed one-form (by Arnaud's Theorem \ref{arnaud}) and that no cohomologous closed one-form has image contained in $\{z\in T^*M \mid H(z)<c\}$ (by the conclusion of Theorem \ref{bound-rig}). Moreover, in this case the image of $\Lambda$ by the Legendre transform $T^*M \rightarrow TM$ associated to $H$ is the Aubry set corresponding to the cohomology class $\eta$. This follows from Fathi's characterisation of the Aubry set (see \cite{FS04}[Theorem 1.4] and also \cite{Falibro}[Section 8.5], \cite{FGS}[Section 3]).}
\end{rem}

According to Theorem \ref{bound-rig}, whenever there exists a Lagrangian submanifold $\Lambda \in \mathcal{L}(T^*M)$ which lies in an optical hypersurface $\Sigma$ and whose dynamics is strongly chain recurrent, we cannot find any other one \textit{inside} $\overline{U_{\Sigma}}$ with the same Liouville class. However, we could find one \textit{outside} $U_{\Sigma}$, as this example shows. 

\begin{figure}[here]
\begin{center}
\scalebox{.3}{\includegraphics
{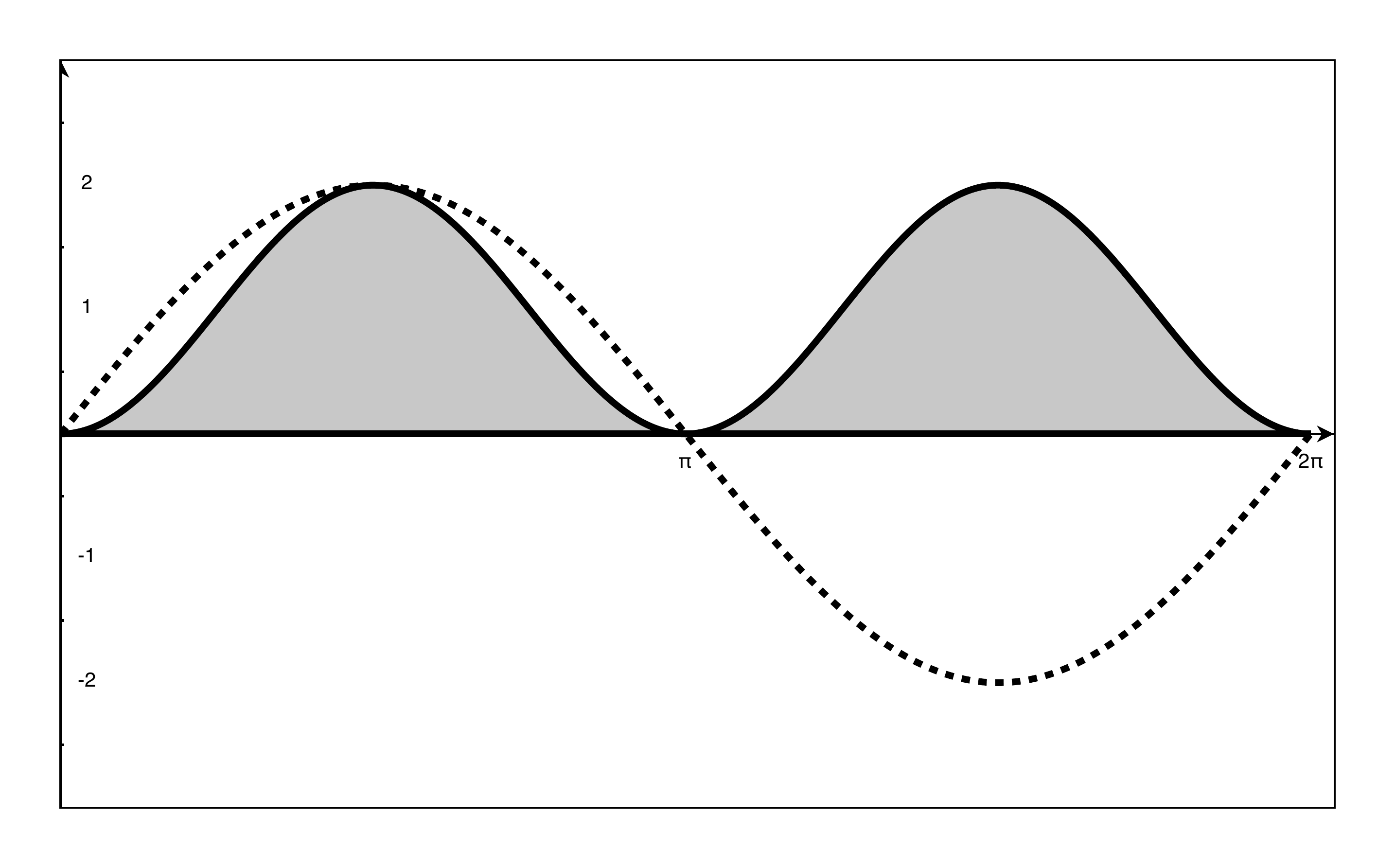}}
\caption{Deformations outside $U_{\Sigma}$ are possible for the Tonelli Hamiltonian (\ref{primo esempio})}
\end{center}
\end{figure}

\begin{ex} \label{no rigidity} 
\textnormal{Let $H: T^*\T^2 \to \R$ be the Tonelli Hamiltonian
\begin{equation} \label{primo esempio}
H(x_1,x_2,y_1,y_2) = y_1^2 + y_2^2 - (1-\cos 2x_1)y_1 - y_2.
\end{equation}
It admits $0$ as regular value and we consider the exact Lagrangian  submanifold given by the zero section $\mathcal{O}$ of $T^*\mathbb{T}^2$, which lies in $\Sigma = H^{-1}(0)$. The dynamics on $\mathcal{O}$ is described by the decoupled system
\[
\begin{cases}
\dot{x}_1 = -(1-\cos 2x_1), \\
\dot{x}_2 = -1.
\end{cases}
\]
The fact that both right-hand sides of these equations do not change sign and that their zeros are isolated easily implies that the restriction of the dynamics on $\mathcal{O}$ is strongly chain recurrent.
However, taking for example 
\[
u(x_1, x_2) = -2\cos x_1,
\] 
we easily check that the exact Lagrangian $du(\T^2)$ lies outside $U_{\Sigma}$. Indeed 
\begin{eqnarray*}
H(x_1,x_2,\partial_{x_1}u,\partial_{x_2}u) &=& H(x_1,x_2,2\sin x_1,0) \\
&=& 4 \sin^2 x_1 - 2 \sin x_1(1 - \cos 2x_1) \\  
&=& 4 \sin^2 x_1(1 - \sin x_1) \ge 0
\end{eqnarray*}
We refer also to Figure 1 where the set
\[
\{(x_1,y_1) \in \T \times \R \mid y_1^2 - (1-\cos 2x_1)y_1 = 0\}
\]
and the related sub- and super-level sets
\[
\{(x_1,y_1) \in \T \times \R \mid y_1^2 - (1-\cos 2x_1)y_1 \le 0\}
\]
and
\[
\{(x_1,y_1) \in \T \times \R \mid y_1^2 - (1-\cos 2x_1)y_1 \ge 0\}
\]
are indicated by the unbroken line and the colors grey and white respectively. Moreover, 
\[
du(\T^2) = \{(x_1,x_2,2\sin x_1,0) \mid (x_1,x_2) \in \T^2 \}
\]
is here represented by a dotted line. This example can be easily generalized to higher dimension.}
\end{ex}

\section{Outer infinitesimal rigidity}

Let $\Sigma:=H^{-1}(c)\subset T^*M$ be an optical hypersurface, and let 
\[
U_{\Sigma} := \{z \in T^*M \mid H(z) < c \}
\]
be the precompact open set bounded by $\Sigma$. Let $\Lambda \in \mathcal{L}(T^*M)$ be a Lagrangian submanifold lying in $\Sigma \subset T^*M$. As we have seen in Example \ref{no rigidity}, it is in general possible to find Lagrangian submanifolds in $\mathcal{L}(T^*M)$ which have the same Liouville class of $\Lambda$ and are fully contained in $U_{\Sigma}^c = T^*M \setminus U_{\Sigma}$, even when the dynamics of the restriction of $\psi^H$ to $\Lambda$ is strongly chain recurrent. In this example, however, it is not possible to deform continuously $\Lambda$ in the set $U_{\Sigma}^c$ while keeping the Liouville class constant.

The next example shows a case in which this continuous deformation is possible, with the dynamics on $\Lambda$ being recurrent.

\begin{figure}[here]
\centerline{\scalebox{.3}{\includegraphics{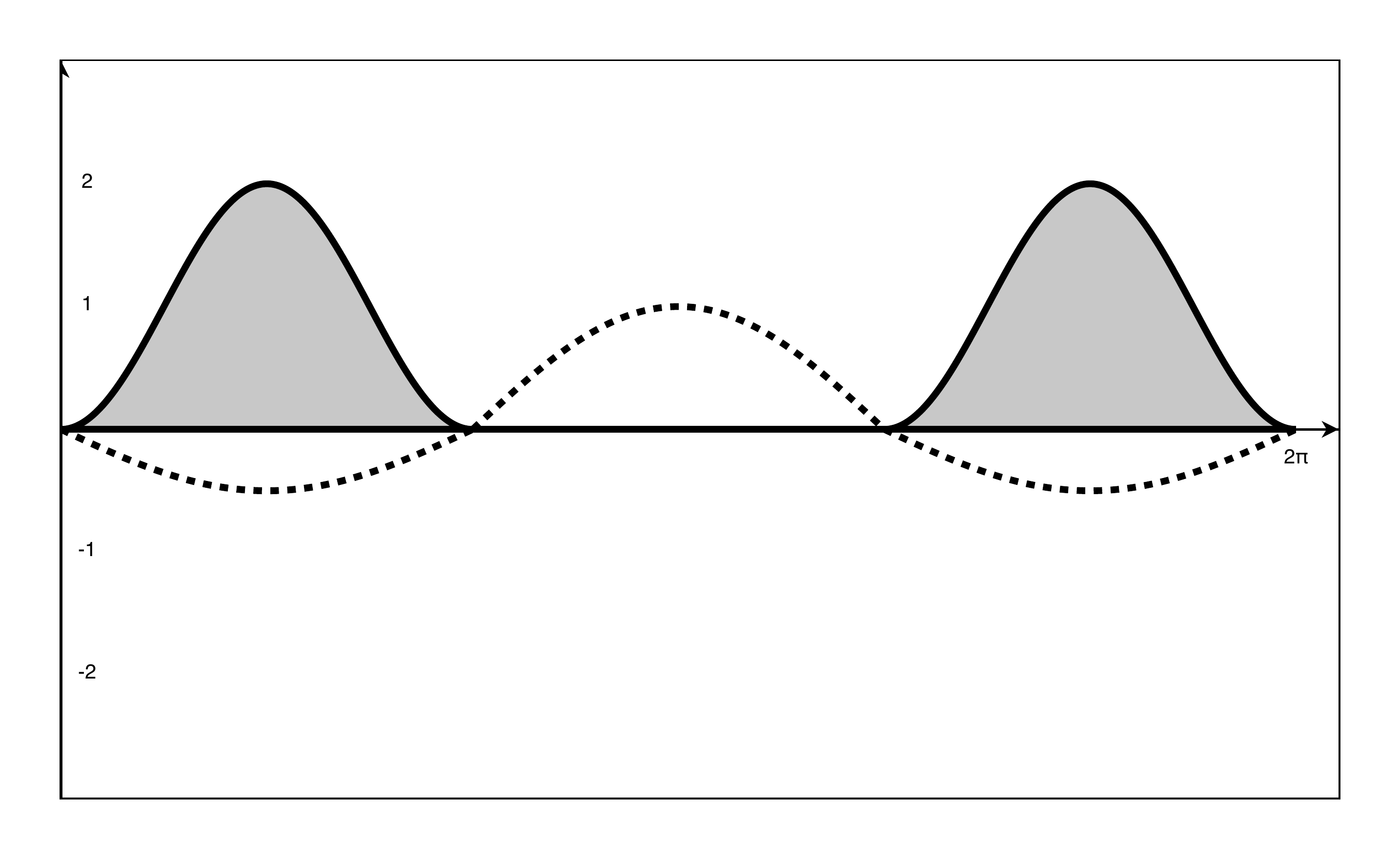}}}
\caption{Sub- and super-level sets for the Tonelli Hamiltonian (\ref{TEe}). An exact Lagrangian graph deforming $\mathcal{O}$ outside $U_{\Sigma}$.} 
\end{figure}

\begin{ex} \label{terzo esempio} 
\textnormal{Let $H:T^*\T^2 \to \R$ be the Tonelli Hamiltonian
\begin{equation}
\label{TEe}
H(x_1,x_2,y_1,y_2) := y_1^2 + y_2^2 - f(x_1)y_1 - y_2,
\end{equation}
where $f:\T \rightarrow \R$ is a non-negative function whose set of zeros is the interval $[1/3,2/3]$. We refer to Figure 2, where the set
\[
\{(x_1,y_1) \in \T \times \mathbb{R} \mid \ y_1^2 - f(x_1)y_1 = 0\}
\]
and the related sub- and super-level sets correspond to the unbroken line and the colors grey and white respectively.
\\
This Hamiltonian admits $0$ as a regular value and the zero section $\mathcal{O}$ of $T^*\mathbb{T}^2$ lies in $\Sigma = H^{-1}(0)$. The dynamics on $\mathcal{O}$ is given by
\[
\begin{cases}
\dot{x}_1 = -f(x_1), \\
\dot{x}_2 = -1.
\end{cases}
\]
The fact that the right-hand sides of the above two equations do not change sign easily implies that every point in $\mathcal{O}$ is chain recurrent. Here, $\mathcal{O}$ can be deformed inside $\widehat{U}_{\Sigma}$ by a non-contstant analytic one-parameter family of exact Lagrangian graphs, see Figure 2. }
\end{ex}
 
The next theorem - a restatement of Theorem \ref{main3} from the Introduction - says that if the dynamics on $\Lambda$ is strongly chain transitive, then every analytic one-parameter deformation of $\Lambda$ within $U_{\Sigma}^c$ with constant Liouville class must be constant. 

\begin{thm} 
Let $\Sigma=H^{-1}(c)$ be an optical hypersurface in $T^*M$. Let $\Lambda$ be an element of $\mathcal{L}(T^*M)$ which is fully contained in $\Sigma$ and such that the restricted Hamiltonian flow
\[
\psi^H|_{\R \times \Lambda} : \R \times \Lambda \rightarrow \Lambda
\]
is strongly chain transitive. Let $\{\Lambda_r\}_{r\in [0,1]}\subset \mathcal{L}(T^*M)$ be an analytic one-parameter family of smooth Lagrangian submanifolds having the same Liouville class of $\Lambda$, such that $\Lambda_0=\Lambda$ and $\Lambda_r \subset U_{\Sigma}^c$ for all $r\in [0,1]$. Then $\Lambda_r = \Lambda$ for all $r\in [0,1]$.
\end{thm}

\begin{proof}
Arguing as at the beginning of the proof of Theorem \ref{bound-rig}, we may assume that $c=0$, that $\Lambda$ coincides with the zero-section $\mathcal{O}$ and that the Hamiltonian has the form 
\[
H(x,y) = \langle y,Y(x) \rangle + F(x,y),
\]
where $Y$ is a smooth vector field on $M$ and $F$ is a smooth non-negative function on $T^*M$ which vanishes on the zero-section. Being non negative, $F$ vanishes up to first order on the zero section. The restriction of the Hamiltonian flow to $\Lambda= \mathcal{O}$ is the flow of the vector field $Y$, after identifying $\mathcal{O}$ with $M$.

The fact that the exact Lagrangian $\Lambda_r$ converges in the $C^1$ topology to the zero section for $r\rightarrow 0$ implies that there exists $r_0\in (0,1]$ such that $\Lambda_r = du_r(M)$ for some analytic one-parameter family of smooth functions $\{u_r\}_{r\in [0,r_0)}$, where $u_0=0$. Since all the $\Lambda_r$'s are contained in 
\[
U_{\Sigma}^c = \{(x,y)\in T^*M \mid H(x,y)\geq 0\},
\] 
we find
\begin{equation}
\label{pos}
0 \leq H(x,du_r(x)) = du_r(x)[Y(x)] + F(x,du_r(x)) \qquad \forall x\in M, \; \forall r\in [0,r_0).
\end{equation}
If, arguing by contradiction, $\Lambda_r$ does not coincide with $\Lambda$ for all $r\in [0,1]$, then the analytic family of functions $u_r$ has the form
\[
u_r = r^h v(x) + r^{h+1} w_r,
\]
for some integer $h\geq 1$, some non-constant smooth function $v$ and some analytic one-parameter family of smooth functions $\{w_r\}_{r\in [0,r_0)}$.
 By plugging the above form into (\ref{pos}) we obtain
\[
0 \leq r^h dv(x)[Y(x)] + F\bigl(x,r^h (dv(x) + r\,dw_r(x))\bigr) \qquad \forall x\in M, \; \forall r\in [0,r_0).
\]
By dividing by $r^h$ we get
\[
dv(x)[Y(x)] \geq - r^{-h} F\bigl(x,r^h (dv(x) + r\,dw_r(x))\bigr),
\]
and by taking a limit for $r\rightarrow 0$, using the fact that $F$ vanishes up to order one on the zero section, we conclude that
\[
dv(x)[Y(x)] \geq 0 \qquad \forall x\in M.
\]
Therefore, $-v$ is a non-constant smooth Lyapunov function for the flow of $Y$. By Corollary \ref{coro}, the flow of $Y$ is not strongly chain transitive (here the easier implication of Proposition \ref{ppp1} suffices). This contradicts our assumption and concludes the proof. 
\end{proof}


\providecommand{\bysame}{\leavevmode\hbox to3em{\hrulefill}\thinspace}
\providecommand{\MR}{\relax\ifhmode\unskip\space\fi MR }
\providecommand{\MRhref}[2]{%
  \href{http://www.ams.org/mathscinet-getitem?mr=#1}{#2}
}
\providecommand{\href}[2]{#2}

\end{document}